\documentclass[12 pt]{article}%
\usepackage{amsmath, amsfonts, amsthm, color,latexsym}
\usepackage{amsmath}
\usepackage{amsfonts}
\usepackage{amssymb}
\usepackage{color, soul}
\usepackage{enumerate}
\providecommand{\U}[1]{\protect\rule{.1in}{.1in}}
\newtheorem{theorem}{Theorem}[section]
\newtheorem{proposition}[theorem]{Proposition}
\newtheorem{corollary}[theorem]{Corollary}
\newtheorem{example}[theorem]{Example}

\newtheorem{remark}[theorem]{Remark}

\newtheorem{final remark}[theorem]{Final Remark}
\newtheorem{definition}[theorem]{Definition}
\allowdisplaybreaks[4]

\begin{document}
	
	%\title{\sc Weak and Weak* Dunford Pettis Multilinear Operators}
	\title{\sc On the transformation of unbounded convergent sequences by linear operators on Banach lattices}
	\author{Fabr\'icio Santos, Ewerton Torres }
	\date{}
	\maketitle

	\begin{abstract}

		In this paper, we introduce a weak variant of the concept of unbounded Dunford–Pettis linear operators, which we call unbounded weak Dunford–Pettis operators. We investigate the properties of this new class, establish inclusion and coincidence results involving unbounded Dunford–Pettis operators, and examine its relationship with the class of unbounded absolutely weakly continuous operators.

%		A continuous operator $T$ between two Banach lattices $E$ and $F$. We say that $T$ is unbounded Dunford-Pettis  whenever $x_{\alpha} \xrightarrow{uaw} 0$ in $E$ implies $T(x_{\alpha}) \xrightarrow{un}$ 0. We say that $T$ is unbounded weak Dunford-Pettis whenever $x_{\alpha} \xrightarrow{uaw} x$ in $E$ and $f_{\alpha} \xrightarrow{w} f$ in $F^*$ implies $f_{\alpha}\left(|T(x_{\alpha}| \wedge u\right) \to f(|T(x)|\wedge u)$, for every $u \in F^+$. In this manuscript, we study some properties of these classes of operators and investigate their relationships with the other classes of operators alread studied in literature. In the third section, we will introduce a new classification of operators named as un-continuous operators and we will investigate on some properties of them and their relationships with other classifications of operators.

	\end{abstract}
	
	\section{Introduction and background}

Let $E$ and $F$ be two normed vector lattices. In the second section of this manuscript, we study on continuous operators $T\colon E \to F$ which carry every norm bounded and unbounded absolutely weak convergent net into a unbounded norm convergent net. The class of these operators is called the class of unbounded Dunford-Pettis continuous operators between two normed vector lattices is denoted by $\mathcal{L}_{uDP}$ (see, \cite{OCCW19}). In the third section, we will introduce a new classification of operators named as unbounded weak Dunford-Pettis operators and we will investigate on some properties of them and their relationships with other classifications of operators. This new class will be denoted by $\mathcal{L}_{uwDP}$.It is worth noting that the main motivation for constructing this new class stems from the ideas presented regarding weakly Dunford-Pettis operators and demi Weakly  Dunford-Pettis (see \cite{AB85} and \cite{RS26}).

To state our results, we need to fix some notations and recall some definitions. Throughout this paper,
the subset $E^+ = \{x \in E\text{; } x \ge 0\}$ of vector lattice $E$ is called the positive cone of $E$ and the elements of $E^+$ are called the positive elements of $E$. If $E$ is a normed space, then $E^*$ is the topological dual space of $E$ and $T^*\colon F^* \to E^*$ is the topological adjoint of continuous operator $T\colon E \to F$ between two normed spaces.

Let $E$ be a vector lattice. We say that the net $(x_{\alpha})_{\alpha} \subset E$ is
\begin{itemize}
\item {\bf weak convergent} to $x \in E$ if $x_{\alpha} \to x$ in $\sigma(E; E^*)$ We denote this convergence by $x_{\alpha} \xrightarrow{w} x$.
\item {\bf unbounded absolutely weakly convergent} to $x\in E$ if $|x_{\alpha} - x|\wedge u \xrightarrow{w} 0$ for all $u \in E^+$. We denote this convergence by $x_{\alpha} \xrightarrow{uaw} x$ (see, \cite{Z18}).
\item {\bf unbounded norm convergent} to $x\in E$ if $\||x_{\alpha} - x|\wedge u\| \to 0$ for all $u \in E^+$. We denote this convergence by $x_{\alpha} \xrightarrow{un} x$ (see \cite{DOT17, KMT17, T04}).
\end{itemize}

The Dunford–Pettis operators, a term coined by A. Grothendieck in \cite{G53}, can be regarded as a generalization of compact linear operators between Banach spaces. Aliprantis and Burkinshaw \cite{AB85} introduced weakly Dunford–Pettis operators, a weak variant of Dunford–Pettis operators that defines a broader class of operators. Subsequently, Benkhaled, Hajji, and Jeribi \cite{BHJ22} introduced demi Dunford–Pettis operators, which also generalize Dunford–Pettis operators. Building on these developments, Ribeiro and Santos \cite{RS26} introduced demi weakly Dunford–Pettis operators as a weak variant of demi-Dunford–Pettis operators. In a related direction, Ribeiro and Santos \cite{RS26(2)} introduced weakly Dunford–Pettis multilinear operators, extending their earlier work on Dunford–Pettis multilinear mappings in \cite{RST22}.

Motivated by these developments, the main objective of this paper is to introduce a new class of linear operators between Banach lattices, which we call unbounded weakly Dunford–Pettis operators. This class provides a weak variant of the class of unbounded Dunford–Pettis operators, introduced in \cite{OCCW19}. We also establish inclusion and coincidence results involving unbounded Dunford–Pettis operators.

Let $X$ and $Y$ be two Banach spaces. An operator $T\colon X \to Y$, is said to be Dunford-Pettis whenever $x_j \to 0$ in $\sigma(X; X^*)$ $($we denote this convergence by $x_j \xrightarrow{w} 0$$)$ implies $\|T(x_j)\| \to 0$. Let $E$ and $F$ be two vector lattices. An operator $T\colon E \to F$, is said to be disjointness preserving whenever $x \perp y$ in $E$ implies $Tx \perp Ty$ in $F$. If for an operator $T\colon E \to F$ between two vector lattices, $T \vee (-T)$ exists, we say its modulus $|T|$ exists. A vector $e > 0$ in vector lattice $E$ is a strong unit when the ideal $I_e$ (generated by $e$) is equal to $E$; equivalently, for every $x \ge 0$ there exists $n \in \mathbb{N}$ such that $x \le ne$. A norm on a Banach lattice $E$ is called order continuous if $\|x_{\alpha}\| \to 0$ for all net $(x_{\alpha})_{\alpha} \subset E$ with $x_{\alpha} \downarrow 0$ and we say $E$ has weakly sequentially continuous  lattices operations whenever $x_j \xrightarrow{w} 0$ implies $|x_j| \xrightarrow{w} 0$.

An element $e \in E^+$ is called an atom of a Riesz space $E$ if the principal ideal $I_e$ is one-dimensional. $E$ is called an atomic Banach lattice if it is the band generated by its atoms. Clearly, $c_0$ and $\ell_p$ ($1 \le p < \infty$) are atomic Banach lattices.

\section{Unbounded Dunford-Pettis Operators on Banach Lattices}

The concept of unbounded Dunford-Pettis operators was initially introduced in \cite{OCCW19}, where the aim was to explore this class and relate it to others introduced in the literature, as well as to new classes introduced in the same paper. Let us recall this concept. In this paper

\begin{definition}
An continuous operator $T$ from a Banach lattice $E$ into a Banach
lattice $F$ is a unbounded Dunford-Pettis operator $($uDP, for short$)$ whenever
\begin{equation*}
T(x_{\alpha}) \xrightarrow{un} 0
\end{equation*}
in $F$ for any bounded net $x_{\alpha} \xrightarrow{uaw} 0$ in $E$.
\end{definition}

In particular,
\begin{definition}
An continuous operator $T$ from a Banach lattice $E$ into a Banach
lattice $F$ is a $\sigma$-unbounded Dunford-Pettis operator $($$\sigma$-uDP, for short$)$ whenever
\begin{equation*}
T(x_n) \xrightarrow{un} 0
\end{equation*}
in $F$ for any bounded sequence $x_n \xrightarrow{uaw} 0$ in $E$.
\end{definition}

The class of all unbounded Dunford-Pettis and $\sigma$-unbounded Dunford-Pettis operators of $\mathcal{L}(E, F)$ will be denoted by $\mathcal{L}_{uDP}(E; F)$ and $\mathcal{L}_{\sigma-uDP}(E; F)$, respectively. When $E = F$ we write only $\mathcal{L}_{uDP}(E)$ and $\mathcal{L}_{\sigma-uDP}(E)$, respectively. We already know from \cite{OCCW19} that in general
\begin{equation*}
\mathcal{L}_{uDP}(E; F) \subset \mathcal{L}_{\sigma-uDP}(E; F).
\end{equation*}
%and that, if $E$ has norm order continuous then 
%\begin{equation*}
%\mathcal{L}_{uDP}(E) = \mathcal{L}_{\sigma-uDP}(E) = \mathcal{L}(E).
%\end{equation*}

An operator $T$ between two Banach lattices $E$ and $F$ is said to be unbounded
norm continuous (or, un-continuous for short) whenever $x_{\alpha} \xrightarrow{un} 0$ implies
$T(x_{\alpha}) \xrightarrow{un} 0$, for $(x_{\alpha})_{\alpha} \subset E$. For sequence $(x_j)_{j=1}^{\infty} \subset E$, if $x_j \xrightarrow{un} 0$ implies
$T(x_j) \xrightarrow{un} 0$, then $T$ is called a $\sigma$-unbounded norm continuous operator (or,
$\sigma$-un-continuous for short). For more details about this class, we reference \cite{MAA19}. The collection of all unbounded norm continuous operators of $\mathcal{L}(E, F)$ will be denoted by $\mathcal{L}_{un}(E, F)$, , and its sequential case by $\mathcal{L}_{\sigma-un}(E, F)$.. It is simple to see that $\mathcal{L}_{uDP}(E; F) \subset \mathcal{L}_{un}(E; F)$.

 %A continuous operator $T\colon E \to F$ between two Banach lattices is said ($\sigma$-)unbounded
%norm continuous (($\sigma$-)un-continuous, for short), whenever $x_{\alpha} \xrightarrow{un} 0$ ($x_j \xrightarrow{un} 0$) in $E$ implies $T(x_{\alpha}) \xrightarrow{un} 0$ ($T(x_j) \xrightarrow{un} 0$) in $F$, for more details see \cite{WC22}.

%\begin{definition}
%Let $E$ be a Banach lattice. 
%\begin{enumerate}
%\item[$(a)$] We say that $E$ has the unbounded Schur property when
%\begin{equation*}
%x_{\alpha} \xrightarrow{uaw} 0 \iff x_{\alpha} \xrightarrow{un} 0.
%\end{equation*}
%\item[$(b)$] We say that $E$ has the $\sigma$-unbounded Schur property when
%\begin{equation*}
%x_n \xrightarrow{uaw} 0 \iff x_n \xrightarrow{un} 0.
%\end{equation*}
%as $n \to \infty$.
%\end{enumerate}
%\end{definition}

%The first result we will present regarding this class is the following:

It follows immediately from \cite[Theorem $4$]{Z18} the following result.

\begin{proposition}\label{unuDP}
Let $E$ be a Banach lattice. Then the following statements are equivalent:
\begin{enumerate}
\item[$(a)$] Every un-continuous operator $T \in \mathcal{L}(E)$ is a unbounded Dunford-Pettis;
\item[$(b)$] The identity $Id_E \colon E \to E$ is a unbounded Dunford-Pettis operator;
\item[$(c)$] $E$ has norm order continuous.
\end{enumerate}
\end{proposition}

Since $\ell_1$ has norm order continuous, by Proposition \ref{unuDP}, $Id_{\ell_1} \in \mathcal{L}_{uDP}(\ell_1)$, however, the $Id_{\ell_1}^{*} = Id_{\ell_{\infty}}\colon \ell_{\infty} \to \ell_{\infty}$ is not a unbounded Dunford-Pettis. Indeed, note that $e_j \xrightarrow{uaw} 0$ in $\ell_{\infty}$ and $(1, 1, 1,\dots,) \in \ell_{\infty}^+$, thus
\begin{equation*}
\||Id_{\ell_{\infty}}(e_j)|\wedge (1, 1, 1,\dots)\|_{\infty} = \||e_j|\wedge (1, 1, 1,\dots)\|_{\infty} = \|e_j\|_{\infty} =1 \not\rightarrow 0. 
\end{equation*}

%It follows immediately from \cite[Theorem $4$]{Z18} that if $E$ has norm order continuous then $Id_E$ is unbounded Dunford-Pettis, and thus, $E$ has  the unbounded Schur property. 

As an immediate consequence of \cite[Proposition $2.5.23$]{MN91}, \cite[Theorem $2.3$]{KMT17} and \cite[Theorem $7$]{Z18}, we can see that the following proposition holds.

\begin{proposition}
Let $E$ and $F$ be Banach lattices. Then
\begin{enumerate}
\item[$(a)$] If $E^*$ has norm order continuous then every Dunford-Pettis operator $T\colon E \to F$  is $\sigma$-unbounded Dunford-Pettis operator.
\item[$(b)$] If $E$ has weakly sequentially continuous  lattices operations and $F$ is a Banach lattice with strong unit then every $\sigma$-unbounded Dunford-Pettis operator $T\colon E \to F$ is Dunford-Pettis operator.
\end{enumerate}
\end{proposition}

It is important to note that the conditions required in the previous proposition do not guarantee reflexivity, which tells us that this result is valid for conditions stronger than reflexivity; however, if we restrict ourselves only to $\sigma$-un-continuous operators, it was shown in\cite[Theorem 4.4]{OCCW19}  that the stataments are equivalents
\begin{enumerate}
\item[$(a)$] $E$ is reflexive;
\item[$(b)$] $T\colon E \to F$ is Dunford-Pettis operator if, and only if, is unbounded Dunford-Pettis operator.
\end{enumerate}

\begin{proposition}\label{uDP+DP = uDP}
	Let $E$ and $F$ be  Banach lattices with the norm of $E^*$ order continuous and  $T \colon E \rightarrow F$ be an $\sigma$-unbounded Dunford-Pettis operator. If $S \colon E \rightarrow F$ is Dunford-Pettis, then the operator $T + S$ is  a  $\sigma$-unbounded Dunford-Pettis operator.
\end{proposition}

\begin{proof}
Let $(x_j)_{j=1}^{\infty}$ be a sequence in $E$ with $x_j \xrightarrow{uaw} 0$. Since $E^*$ has norm order continuous by \cite[Theorem $7$]{Z18}, $x_j \xrightarrow{w} 0$. Thus, $\|S(x_j)\| \to 0$ in $F$ what implies that $S(x_j) \xrightarrow{un} 0$ in $F$. As $T$ is $\sigma$-unbounded Dunford-Pettis, then $T+S \in \mathcal{L}_{\sigma-uDP}(E; F)$.
%\begin{equation*}
%(T+S)(x_j) = T(x_j) + S(x_j) \xrightarrow{un} 0.
%\end{equation*}
\end{proof}

\begin{proposition}\label{uDP+DP = DP}
	Let $E$  be a Banach lattice with  weakly sequentially continuous  lattices operations and $F$ be a Banach lattice with strong order unit. Considere  $T \colon E \rightarrow F$ be an $\sigma$-unbounded Dunford-Pettis operator and $S \colon E \rightarrow F$ be a Dunford-Pettis operator, then the operator $T + S$ is  a  Dunford-Pettis operator.
\end{proposition}

\begin{proof}
Let $(x_j)_{j=1}^{\infty}$ be a sequence in $E$ with  $x_j \xrightarrow{w} 0$. Since, $E$ has weakly sequentially continuous  lattices operations then $x_j \xrightarrow{uaw} 0$ and hence $T(x_j) \xrightarrow{un} 0$ in $F$. As $F$ has strong order unit then $T(x_j) \to 0$. Therefore, since $S$ is a Dunford-Pettis operator then $T+S$ is a Dunford-Pettis operator.
\end{proof}

In the next result, we study a condition for the matrix operator
\begin{equation*}
 \tilde{T} = \left(\begin{matrix}
		T_1 &T_2\\
		T_3 &T_4
	\end{matrix}\right)
\end{equation*}
 to be $\sigma$-unbounded Dunford-Pettis operator, where $\tilde{T}$ is an operator acting on the Banach space $E \times F$, with $E$ and $F$ being Banach lattices, $T_1 \colon E \rightarrow E$, $T_4 \colon F \rightarrow F$, $T_2 \colon F \rightarrow E$ and $T_3 \colon E \rightarrow F$.

 \begin{proposition}
 	Let $E$ be a Banach lattice with norm order continuous and $F$ be a Banach lattice with the norm of $F^*$ order continuous. Consider $\tilde{T} \colon E \times F \rightarrow E \times F$ given by
 	\begin{equation*}
 		\tilde{T} = \left(\begin{matrix}
 			T_1 &T_2\\
 			T_3 &T_4
 		\end{matrix}\right)
 	\end{equation*}
 If $T_1$ and $T_4$ are $\sigma$-unbounded Dunford-Pettis operator and $T_2$ is Dunford-Pettis operator and $T_3$ is a $\sigma$-un-continuous, then $\tilde{T}$ is  $\sigma$-unbounded Dunford-Pettis operator.
 \end{proposition}

\begin{proof}
Let $((x_j, y_j))_{j=1}^{\infty}$ a sequence in $E \times F$ such that $(x_j, y_j) \xrightarrow{uaw} (0, 0)$. It is easy see that $x_j \xrightarrow{uaw} 0$ in $E$ and $y_j \xrightarrow{uaw} 0$ in $F$. Thus,  
\begin{equation*}
T_1(x_j) \xrightarrow{un} 0 \text{ and } T_4(y_j) \xrightarrow{un} 0.
\end{equation*}
Since, $E$ and $F^*$ has norm order continuous then
\begin{equation*}
x_j \xrightarrow{uaw} 0 \Rightarrow x_j \xrightarrow{un} 0 \text{ and } y_j \xrightarrow{uaw} 0 \Rightarrow y_j \xrightarrow{w} 0.
\end{equation*}
Since, $T_2$ is a Dunford-Pettis and $T_3$ is $\sigma$-un-continuous operator then
\begin{equation*}
T_2(y_j) \to 0 \Rightarrow T_2(y_j) \xrightarrow{un} 0 \text{ and } T_3(x_j) \xrightarrow{un} 0.
\end{equation*}
So, for every $(u, v) \in E^+ \times F^+$
\begin{align*}
&\left\||\tilde{T}(x_j, y_j)| \wedge (u, v)\right\|\\
&= \left\||(T_1(x_j) + T_2(y_j), T_3(x_j)+T_4(y_j))| \wedge (u, v)\right\|\\
&= \left\|(|T_1(x_j) + T_2(y_j)|, | T_3(x_j)+T_4(y_j)|) \wedge (u, v)\right\|\\
&\le \left\|(|T_1(x_j)|\wedge u + |T_2(y_j)|\wedge u, |T_3(x_j)|\wedge v + |T_4(y_j)|\wedge v)\right\|\\
&\le \left\||T_1(x_j)|\wedge v \right\| + \left\||T_2(y_j)|\wedge u\right\| + \left\||T_3(x_j)|\wedge v \right\| + \left\||T_4(y_j)|\wedge u\right\| \to 0.
\end{align*}
Then, $\tilde{T}(x_j, y_j) \xrightarrow{un} (0, 0)$. Hence, $\tilde{T}$ is a $\sigma$-unbounded Dunford-Pettis operator.
\end{proof}

\begin{theorem}
Let $E$ be a Banach lattice whose dual $E^*$ has norm order continuous, and let $F$ be an atomic Banach lattice with norm order continuous. Then every bounded linear operator $T\colon E\to F$ is unbounded Dunford-Pettis. Consequently,
\begin{equation*}
\mathcal{L}(E;F)=\mathcal{L}_{uDP}(E;F).
\end{equation*}
\end{theorem}

\begin{proof}
 Consider $T \in \mathcal{L}(E; F)$ and let a bounded net $(x_{\alpha})_{\alpha}$ in $E$ such that $x_{\alpha} \xrightarrow{uaw} 0$. Since $E^*$ has norm order continuous, by \cite[Theorem $7$]{Z18},
\begin{equation*}
x_{\alpha} \xrightarrow{w} 0
\end{equation*}
and as $T\in \mathcal{L}(E; F)$ and $F$ is a atomic order continuous Banach lattice, by \cite[Proposition $6.2$]{DOT17} we have
\begin{equation*}
T(x_{\alpha}) \xrightarrow{un} 0
\end{equation*}
as $j \to \infty$. So, $T \in \mathcal{L}_{uDP}(E; F)$.
\end{proof}

\begin{definition}
An operator $T\colon E \to F$ between two Riesz spaces is called a lattice homomorphism if
\begin{equation*}
T(x \vee y)=T(x)\vee T(y)
\end{equation*}
for all $x,y\in E$.
\end{definition}

\begin{remark}
\begin{enumerate}
\item[$(a)$] It is immediate that every lattice homomorphism $T\colon E \to F$ is a positive and disjointness-preserving operator (see \cite[Theorem 2.14]{AB85}). Moreover,
\begin{equation*}
|T(x)|=T(|x|)
\end{equation*}
for every $x\in E$.
\item[$(b)$] A bijective lattice homomorphism is called a lattice isomorphism.
\end{enumerate}
\end{remark}

Let $E$ and $F$ be Banach lattices and $T \colon E \to F$. We say that T is (sequentially) un-compact if $T(B_E)$ is relatively (sequentially) un-compact in $F$. Equivalently, for every bounded net $(x_{\alpha})_{\alpha}$ (respectively, every bounded sequence $(x_j)_{j=1}^{\infty}$) its image has a subnet (respectively, subsequence), which is un-convergent, for more details see \cite{KMT17}. 

%\begin{remark}
%Note that, considering that $E^*$ has norm order continuous, then every net $(x_{\alpha})_{\alpha}$ $($sequence $(x_j)_{j=1}^{\infty}$$)$ uaw-null is, in particular, bounded, according to \cite[Theorem $7$]{Z18}. So, under this condition, we can see that every un-compact operator $T\colon E \rightarrow F$ is an unbounded Dunford-Pettis operator.
%\end{remark}

%For simplicity, we denote by $\mathcal{L}_{u\mathcal{K}}$ the class of all un-compact operators between Banach lattices. 

\begin{theorem}
Let $E$ and $F$ ba Banach lattices and $T \in \mathcal{L}(E; F)$.
\begin{enumerate}
\item[$(a)$] If $T$ is a lattice isomorphism sequentially un-compact then  $T$ is a $\sigma$-unbounded Dunford-Pettis operator.
\item[$(b)$] If $E$ is a reflexive and atomic Banach lattice and $T$  is a $\sigma$-unbounded Dunford-Pettis operator then $T$ is a   sequentially un-compact operator.
\end{enumerate}
\end{theorem}

\begin{proof}
$(a)$ Let $T\colon E\to F$ be a lattice isomorphism that is sequentially un-compact, and let $(x_j)$ be a bounded sequence in $E$ such that $x_j\xrightarrow{uaw}0$.

By the sequential un-compactness of $T$, there exists a subsequence $(T(x_{j_k}))$ such that
\begin{equation*}
T(x_{j_k})\xrightarrow{un}y
\end{equation*}
for some $y\in F$. Since un-convergence implies uaw-convergence, we also have
\begin{equation*}
T(x_{j_k})\xrightarrow{uaw}y.
\end{equation*}

On the other hand, since $T$ is a lattice isomorphism, for every $v\in F^+$ there exists $u\in E^+$ such that $v=T(u)$. Hence,
\begin{equation*}
\begin{aligned}
|T(x_{j_k})|\wedge v
&=T(|x_{j_k}|)\wedge T(u)\
&=T(|x_{j_k}|\wedge u)
\xrightarrow{w}0,
\end{aligned}
\end{equation*}
where the last convergence follows from
$x_{j_k}\xrightarrow{uaw}0$ and the continuity of $T$.
Thus, $T(x_{j_k})\xrightarrow{uaw}0$.
By the uniqueness of the uaw-limit, $y=0$.

Consequently, every subsequence of $(T(x_j))$ has a further subsequence that converges un to zero. It follows that $T(x_j)\xrightarrow{un}0$. Indeed, otherwise, there would exist $u\in F^+$, $\varepsilon>0$, and a subsequence $(T(x_{j_k}))$ such that
\begin{equation*}
\||T(x_{j_k})|\wedge u\|\geq\varepsilon
\end{equation*}
for every $k$. By the sequential un-compactness of $T$, this subsequence would have a further subsequence un-convergent  to some $z\in F$. The preceding argument yields $z=0$, contradicting the displayed inequality.
Therefore, $T$ is a $\sigma$-unbounded Dunford--Pettis operator.

$(b)$ Suppose, for contradiction, that $T$ is not an sequentially un-compact operator, then there exists a bounded sequence $(x_j)_{j=1}^{\infty}$ in $E$ such that $(T(x_j))_{j=1}^{\infty}$ has no un-convergent subsequence. Since $(x_j)_{j=1}^{\infty}$ is bounded and $E$ is reflexive, by the Eberlein–Šmulian Theorem, there exists a subsequence $(x_{j_k})_{k=1}^{\infty}$ that is weakly convergent. As $E$ is an atomic order continuous Banach lattice, by \cite[Proposition $2.5.23$]{MN91}, the lattice operations are weakly sequentially continuous, and thus $(x_{j_k})_{k=1}^{\infty}$ is uaw-convergent. Since $T$ is a $\sigma$-unbounded Dunford-Pettis operator then $(T(x_{j_k}))_{k=1}^{\infty}$ is un-convergent. But this does not occur.
\end{proof}

%\begin{theorem}
%Let $E$ be a atomic order continuous Banach lattice and $F$ be a Banach lattice and $T \in \mathcal{L}(E; F)$. Then, $T$ is $\sigma$-un-compact if, and only if, T is $\sigma$-unbounded Dunford-Pettis
%\end{theorem}

%\begin{proof}
%"$\Rightarrow$" It has already been addressed in the previous observation.

%"$\Leftarrow$" Suppose, for contradiction, that $T$ is not an unbounded compact operator, then there exists a bounded sequence $(x_j)_{j=1}^{\infty}$ in $E$ such that $(T(x_j))_{j=1}^{\infty}$ has no un-convergent subsequence. Since $(x_j)_{j=1}^{\infty}$ is bounded, by the Eberlein–Šmulian Theorem, there exists a subsequence $(x_{j_k})_{k=1}^{\infty}$ that is weakly convergent. As $E$ is an atomic order continuous Banach lattice, by \cite[Proposition $2.5.23$]{MN91}, the lattice operations are weakly sequentially continuous, and thus $(x_{j_k})_{k=1}^{\infty}$ is uaw-convergent. Since $T$ is a unbounded Dunford-Pettis then $(T(x_{j_k}))_{k=1}^{\infty}$ is un-convergente. But this does not occur.
%\end{proof}

In \cite[Proposition $3.1$]{OCCW19}, it was shown that if $T, S\colon E \to F$ with $T$ an unbounded Dunford-Pettis operator and $0 \le S \le T$, then $S$ is also an unbounded Dunford-Pettis operator. Consider $E, F$ and $G$ be Banach lattices and
\begin{equation*}
S_1, T_1 \colon E \to F \text{ and } S_2, T_2 \colon F \to G
\end{equation*}
positive operators with $0 \le S_i \le T_i$ ($i=1, 2$). If $T_1$ and $T_2$ are unbounded Dunford-Pettis operators, then $S_2 \circ S_1$ is an unbounded Dunford-Pettis operator. Indeed, since $T_1$ and $T_2$ are unbounded Dunford-Pettis operators then $S_1$ and $S_2$  are also unbounded Dunford-Pettis and for every $x_{\alpha} \xrightarrow{uaw} 0$ we have
\begin{equation*}
S_1(x_{\alpha}) \xrightarrow{un} 0
\end{equation*}
and, consequently
\begin{equation*}
S_1(x_{\alpha}) \xrightarrow{uaw} 0.
\end{equation*}
Thus,
\begin{equation*}
S_2(S_1(x_{\alpha})) \xrightarrow{un} 0.
\end{equation*}
Hence, $S_2\circ S_1 \in \mathcal{L}_{uDP}(E; G)$.

Following this same line of thought, we can see that if $S \in \mathcal{L}_{uDP}(E; F)$ and $T \in \mathcal{L}_{un}(F; G)$ then $T \circ S \in \mathcal{L}_{uDP}(E; G)$. Indeed, let $(x_{\alpha})_{\alpha}$ in $E$ with $x_{\alpha} \xrightarrow{uaw} 0$, then $S(x_{\alpha}) \xrightarrow{un} 0$. Thus,
\begin{equation*}
T \circ S(x_{\alpha}) = T(S(x_{\alpha})) \xrightarrow{un} 0.
\end{equation*}
Hence, $T \circ S \in \mathcal{L}_{uDP}(E; G)$.

It is also important to note that, in general, there is no hope for the class $\mathcal{L}_{uDP}$ to be an ideal of linear operators (\cite{P67}), since unbounded absolutely weakly convergent sequences are not closed under linear functionals. Indeed, consider
\begin{equation*}
E = \ell_1 \text{ and } x_j = e_j.
\end{equation*}
 It is simple to see that $e_j \xrightarrow{uaw} 0$, however, for $f = (1, 1, 1,\dots) \in \ell_{\infty} = \ell_1^*$ we have
\begin{equation*}
f(e_j) = 1 \not\rightarrow 0.
\end{equation*}
This highlight a fundamental difference between unbounded absolutely weak convergence and weak convergence.

\begin{proposition}
Let $E$ and $F$ be Banach lattices and $T \in \mathcal{L}(E; F)$. If $T \in \mathcal{L}_{uDP}(E; F)$ and preserves disjointness then $|T|$ exists and belongs to $\mathcal{L}_{uDP}(E; F)$.
\end{proposition}

\begin{proof}
By \cite[Theorem $2.40$]{AB85}, $|T|$ exists and for all $x \in E$
\begin{equation*}
|T(x)| = |T(|x|)| = |T|(|x|).
\end{equation*} 
Let $(x_{\alpha})_{\alpha} \subset E$ with $x_{\alpha} \xrightarrow{uaw} 0$,  for all $u \in F^+$ we have
\begin{equation*}
||T|(x_{\alpha})|\wedge u \le ||T|(|x_{\alpha}|)|\wedge u  = |T(x_{\alpha})|\wedge u.
\end{equation*}
Since, $T \in \mathcal{L}_{uDP}(E; F)$, the inequality above shows us that $|T| \in \mathcal{L}_{uDP}(E; F)$.
\end{proof}

\section{Unbounded weak Dunford-Pettis Operators on Banach Lattices}

In this section, we introduce the concept of unbounded weak Dunford-Pettis operators and explore their properties, in particular, their relationship with the classes of unbounded Dunford-Pettis and un-compact operators.

A natural first attempt at defining unbounded weak Dunford-Pettis operators is the following.

\begin{definition} Let $E$ and $F$ be Banach lattices and $T \in \mathcal{L}(E; F)$. We say that $T$ is an unbounded weak Dunford-Pettis whenever for every bounded nets $(x_{\alpha})_{\alpha}$ in $E$ and $(f_{\alpha})_{\alpha}$ in $(F^*)^+$,  with $x_{\alpha} \xrightarrow{uaw} 0$ and $f_{\alpha} \xrightarrow{w} 0$ then
\begin{equation*}
f_{\alpha}\left(|T(x_{\alpha})| \wedge u \right) \to 0
\end{equation*}
for every $u \in F^+$. In this case, we write $T \in \mathcal{L}_{uwDP}(E; F)$.
%\item[$(b)$] We say that $T$ is a $\sigma$-unbounded weak Dunford-Pettis whenever for every bounded sequences $(x_j)_{j=1}^{\infty}$ in $E$ and $(f_j)_{j=1}^{\infty}$ in $(F^*)^+$,  with $x_j \xrightarrow{uaw} 0$ and $f_j \xrightarrow{w} 0$ then
%\begin{equation*}
%f_j\left(|T(x_j)| \wedge u \right) \to 0
%\end{equation*}
%as $j \to \infty$, for every $u \in E^+$.
%\end{enumerate}
\end{definition}

However, such an approach is not attractive because, for all $T \in \mathcal{L}(E; F)$, bounded nets $(x_{\alpha})_{\alpha}$ in $E$ and $(f_{\alpha})_{\alpha}$ in $(F^*)^+$ with $x_{\alpha} \xrightarrow{uaw} 0$ in $E$ and $f_j \xrightarrow{w} 0$ in $F^*$,  and $u \in F^+$, we have
\begin{equation*}
0 \le |T(x_{\alpha})| \wedge u \le u \Rightarrow 0 \le f_{\alpha}( |T(x_{\alpha})| \wedge u) \le f_{\alpha}(u).
\end{equation*}
Since, $f_{\alpha}(u) \to 0$ then $f_{\alpha}( |T(x_{\alpha})| \wedge u) \to 0$ for all $T \in \mathcal{L}(E; F)$. So, 
\begin{equation*}
\mathcal{L}_{uwDP}(E; F) = \mathcal{L}(E; F).
\end{equation*}

%The class of all unbounded weak Dunford-Pettis between Banach lattices is denoted by  $\mathcal{L}_{uwDP}(E; F)$ and the class of all unbounded weak Dunford-Pettis by $\mathcal{L}_{\sigma-uwDP}(E; F)$. When, $E = F$ we just writh  $\mathcal{L}_{uwDP}(E)$ and $\mathcal{L}_{\sigma-uwDP}(E)$, respectively.

For this reason, we present the following concept for unbounded weak Dunford-Pettis continuous linear operators between Banach lattices.

\begin{definition}
 Let $E$ and $F$ be Banach lattices and $T \in \mathcal{L}(E; F)$.
\begin{enumerate}
\item[$(a)$] $T$ is unbounded weak Dunford-Pettis if,  for every bounded nets $(x_{\alpha})_{\alpha}$ in $E$ and $(f_{\alpha})_{\alpha}$ in $(F^*)^+$,  with $x_{\alpha} \xrightarrow{uaw} x \in E$ and $f_{\alpha} \xrightarrow{w} f \in F^*$ then
\begin{equation*}
f_{\alpha}\left(|T(x_{\alpha})| \wedge u \right) \to f\left(|T(x)|\wedge u \right)
\end{equation*}
for every $u \in E^+$.
\item[$(b)$] $T$ is $\sigma$-unbounded weak Dunford-Pettis if, for every bounded sequences $(x_j)_{j=1}^{\infty}$ in $E$ and $(f_j)_{j=1}^{\infty}$ in $(F^*)^+$,  with $x_j \xrightarrow{uaw} x \in E$ and $f_j \xrightarrow{w} f \in F^*$ then
\begin{equation*}
f_j\left(|T(x_j)| \wedge u \right) \to f\left(|T(x)|\wedge u \right)
\end{equation*}
as $j \to \infty$, for every $u \in E^+$.
\end{enumerate}
\end{definition}
The class of all unbounded weak Dunford-Pettis between Banach lattices is denoted by  $\mathcal{L}_{uwDP}(E; F)$ and the class of all unbounded weak Dunford-Pettis by $\mathcal{L}_{\sigma-uwDP}(E; F)$. When, $E = F$ we just write  $\mathcal{L}_{uwDP}(E)$ and $\mathcal{L}_{\sigma-uwDP}(E)$, respectively.

\begin{example}
The canonical inclusion $J\colon c_0 \to \ell_\infty$, given by $J(x) = x$, is a $\sigma$-unbounded weak Dunford-Pettis operator. Indeed, let $(x_j)_{j=1}^{\infty}$ in $c_0$ and $(f_j)_{j=1}^{\infty}$ in $(\ell_{\infty}^*)^+$ such that $x_j \xrightarrow{uaw} x \in E$ and $f_j \xrightarrow{w} f \in \ell_{\infty}^*$. Since, $J$ is a canonical inclusion then
\begin{equation*}
f_j(|J(x_j)|\wedge u) = f_j(|x_j|\wedge u) \text{ and } f(|J(x)|\wedge u) = f(|x|\wedge u)
\end{equation*}
for any $u = (u_k)_{k=1}^{\infty} \in \ell_{\infty}^+$. Thus,
\begin{equation*}
0 \le |x_j(k)|\wedge u_k \le |x_j(k)|
\end{equation*}
and so, $ |x_j(k)|\wedge u_k \to 0$, as $k\to \infty$. Then, $|x_j|\wedge u \in c_0$. It follows the same ideas, $|x|\wedge u \in c_0$. Since,
\begin{equation*}
||x_j|\wedge u - |x|\wedge u| \le |x_j - x| \wedge u,
\end{equation*}
and given that, $x_j \xrightarrow{uaw} x$, then, $|x_j - x|\wedge u \xrightarrow{w} 0$. So, for every $g \in \ell_{\infty}^*$
\begin{equation*}
|g(|x_j|\wedge u - |x|\wedge u)| \le |g|(||x_j|\wedge u - |x|\wedge u|) \le |g|(|x_j - x| \wedge u).
\end{equation*}
So,
\begin{equation*}
|x_j|\wedge u \xrightarrow{w} |x|\wedge u.
\end{equation*}
Now, consider $g_j = f_j|_{c_0}$ and $g = f|_{c_0}$. We can easily see that $g_j, g \in c_0^* = \ell_1$, $g_j \ge 0$ and $g_j \xrightarrow{w} g$. By Schur property in $\ell_1$, we have
\begin{equation*}
\|g_j - g\|_1 \to 0.
\end{equation*} 
Thus,
\begin{align*}
&|f_j(|x_j|\wedge u) - f(|x|\wedge u)|\\ 
&= |g_j(|x_j|\wedge u) - g(|x|\wedge u)|\\
&\le |(g_j - g)(|x_j|\wedge u)| + |g(|x_j|\wedge u - |x|\wedge u)|\\
&\le \|g_j - g\| \||x_j|\wedge u\|_{\infty} +  |g(|x_j|\wedge u - |x|\wedge u)|.
\end{align*}
So, 
\begin{equation*}
f_j(|x_j|\wedge u)  \to f(|x|\wedge u).
\end{equation*}
Hence, $J \in \mathcal{L}_{\sigma-uwDP}(c_0; \ell_\infty)$.
\end{example}

We can immediately see that, for any Banach lattices $E$ and $F$ 
\begin{equation*}
\mathcal{L}_{uDP}(E; F) \subset \mathcal{L}_{uwDP}(E; F) \subset \mathcal{L}(E; F).
\end{equation*}
In general,
\begin{equation*}
\mathcal{L}_{uwDP}(E; F) \varsubsetneq \mathcal{L}(E; F).
\end{equation*}
Indeed, consider $E = \ell_1$ and define $T\colon \ell_1 \to \ell_1$ given by
\begin{equation*}
T(x) = \varphi(x) e_1
\end{equation*}
where $e_1 = (1, 0,\dots)$ and
\begin{equation*}
\varphi\left((x_k)_{k=1}^{\infty}\right) = \sum_{k=1}^{\infty}x_k.
\end{equation*}
It is immediate that $T\in \mathcal{L}(E; F)$. On the other hand, let $(x_j)_{j=1}^{\infty}$ in $\ell_1$ with $x_j = e_j$, where $e_j = (0,\dots, 0,\underset{j-th}{1}, 0, \dots)$. It is easy see that $x_j \xrightarrow{uaw} 0$. Now, consider $f_j = f = (1, 1,\dots) \in \ell_{\infty} = (\ell_1)^{*}$. Since
\begin{equation*}
\varphi(x_j) = \varphi(e_j) = \sum_{k=1}^{\infty} e_j(k) = 1
\end{equation*}
then, $T(x_j) = T(e_j) = e_1$. For $e_1 \in c_0^+$
\begin{equation*}
f_j\left(|T(x_j)|\wedge e_1 \right) = f_j(e_1) = 1,
\end{equation*}
and thus, 
\begin{equation*}
f_j\left(|T(x_j)|\wedge e_1 \right) \not\rightarrow 0 = f(|T(0)|\wedge e_1).
\end{equation*}

%Vejamos que $x_j \xrightarrow{uaw} 0$. Seja $u = (u_k)_k \in c_0^+$, como para $j$ suficientemente grande temos
%\begin{equation*}
%j^2 \ge \|u\|_{\infty} \ge u_j
%\end{equation*}
%pois $u_j \to 0$, then
%\begin{equation*}
%|x_j|\wedge u = (j^2g_j)\wedge u = \min{j^2, u_j}e_j = u_je_j \xrightarrow{w} 0.
%\end{equation*}
%Vejamos que $T(x) = \varphi(x) e_1 \in \mathcal{L}_{uwDP}(E; F)$. Seja $x_j \xrightarrow{uaw} 0$ in $E$ e $f_j \xrightarrow{w} 0$ in $(F^*)^+$. Assim, para todo $u \in c_0^+$
%\begin{align*}
%f_j\left(|T(x_j)| \wedge u \right) &= f_j\left(|\varphi(x_j)| e_1\wedge u \right)\\
%&= f_j(\min\{|\varphi(x_j)|, u_1\}e_1)\\
%&= \min\{|\varphi(x_j)|, u_1\} f_j(e_1).
%\end{align*} 
%Como,
%\begin{equation*}
%0 \le \min\{|\varphi(x_j)|, u_1\} \le u_1 \Rightarrow 0 \le \min\{|\varphi(x_j)|, u_1\} f_j(e_1) \le u_1 f_j(e_1)
%\end{equation*}
%e $ f_j(e_1) \to 0$ then
%\begin{equation*}
% \min\{|\varphi(x_j)|, u_1\} f_j(e_1) \to 0.
%\end{equation*}

We also can see that
\begin{equation*}
\mathcal{L}_{\sigma-uDP}(E; F) \varsubsetneq \mathcal{L}_{\sigma-uwDP}(E; F).
\end{equation*}
Indeed, Consider the canonical inclusion $J\colon c_0 \to \ell_\infty$ given by $J(x) = x$. We have already seen that $J \in \mathcal{L}_{\sigma-uwDP}(c_0; \ell_\infty)$, we need to see then that $T \notin \mathcal{L}_{uDP}(c_0; \ell_{\infty})$. Let $(e_j)_{j=1}^{\infty}$ in $c_0$, it is simple to see that $e_j \xrightarrow{uaw} 0$ in $c_0$. Note that, $J(e_j) = e_j$ and take $u = (1, 1,\dots) \in \ell_{\infty}^+$, thus
\begin{equation*}
\||J(e_j)|\wedge u\|_{\infty} = \|e_j \wedge u\|_{\infty} = \|e_j\|_{\infty} = 1
\end{equation*}
for any $j \in \mathbb{N}$. So,
\begin{equation*}
\||J(e_j)|\wedge u\|_{\infty} \not\to 0.
\end{equation*}
Hence, $T \notin \mathcal{L}_{\sigma-uDP}(c_0; \ell_{\infty})$.

%$(x_j)_{j=1}^{\infty}$ in $c_0$ such that $x_j = j^3e_j$. It is simple see that $x_j \xrightarrow{uaw} 0$. 
%%pois para qualquer $u \in c_0^+$ e $f \in (c_0)^{*}$ temos
%%\begin{equation*}
%%|f(|x_j|\wedge u)| = |f(j^3e_j\wedge u)| = |f(\min\{j^3, u_j\}e_j)| = (\min\{j^3, u_j\} |f(e_j)| \le u_j |f(e_j)| \to 0.
%%\end{equation*}
%Whenever,
%\begin{equation*}
%T(x_j) = \varphi(x) e_1 = je_1
%\end{equation*}
%and so,
%\begin{equation*}
%\||T(x_j)| \wedge e_1\|_{\infty} = \|je_1 \wedge e_1\|_{\infty} = \|e_1\|_{\infty} = 1 \not\rightarrow 0.
%\end{equation*}

\begin{theorem}
 Let $E$ and $F$ be Banach lattices and $T \in \mathcal{L}(E; F)$. If $F$ is a reflexive Banach lattice and $F^*$ has weakly sequentially continuous  lattices operations then every $T$  $\sigma$-unbounded weak Dunford-Pettis operator is also a $\sigma$-unbounded Dunford-Pettis operator.
\end{theorem}
\begin{proof}
Suppose that $T$ is not a unbounded Dunford-Pettis operator, thus there exists a $(x_j)_{j=1}$ with $x_j \xrightarrow{uaw} 0$ such that $T(x_j) \not\xrightarrow{un} 0$, that is, unless there is a subsequence, we can assume that, there exists $u \in F^+$ and $\epsilon > 0$ such that
\begin{equation*}
\||T(x_j)|\wedge u\| > \epsilon
\end{equation*}
for any $j \in \mathbb{N}$.  By Hahn-Banach Theorem, for any $j \in \mathbb{N}$ there exists $f_j \in F^*$ with $\|f_j\| = 1$ and $f_j(|T(x_j)|\wedge u) = \||T(x_j)|\wedge u\|$. By Banach-Alaoglu Thorem, there exists a subsequence $(f_{j_k})_{k=1}^{\infty}$ such that
\begin{equation*}
f_{j_k} \xrightarrow{w} f
\end{equation*}
as $k \to \infty$. Consider $g_j = |f_j|$, thus $\|g_j\| = 1$ and
\begin{equation*}
g_j(|T(x_j)|\wedge u) = \||T(x_j)|\wedge u\|.
\end{equation*}
%De fato, chamando $v=|T(x_j)|\wedge u$ podemos ver que
%\begin{equation*}
%g_j(v) = |f_j|(v) = sup_{|y| < v}|f_j(y)| \le sup_{|y| < v} \|f_j\| \|y\| = sup_{|y| < v} \|f_j\| \||y|\| = \|f_j\| \|v\| = \|v\|.
%\end{equation*}
%Por outro lado, 
%\begin{equation*}
%g_j(v) = |f_j|(v) \ge |f_j(v)| = \|v\|.
%\end{equation*}
This way, there exists $(g_{j_k})_{k=1}^{\infty}$ a subsequence of $(g_j)_{j=1}^{\infty}$ such that
\begin{equation*}
g_{j_k} \xrightarrow{w} |f|
\end{equation*}
because, $F^*$ has weakly sequentially continuous  lattices operations.
%De fato,  Note que,
%\begin{equation*}
%|g_{j_k} - |f|| = ||f_{j_k}| - |f|| \le |f_{j_k} - f|.
%\end{equation*}
%por uma questão de simplicidade, chamaremos $z_k = g_{j_k} - |f|$ e $h_k = f_{j_k} - f$. Dessa forma, $h_k \xrightarrow{w} 0$ e
%\begin{equation*}
%|z_k| \le |h_k|.
%\end{equation*}
%Seja $\Phi \in E^{**}$, assim
%\begin{equation*}
%|\Phi(z_k)| \le |\Phi|(|z_k|) \le |\Phi|(|h_k|).
%\end{equation*}
%Como, $ |\Phi|(|h_k|) \to 0$ pois $\Phi \in E^{**}$ e  $h_k \xrightarrow{w} 0$ então
%\begin{equation*}
%|\Phi(z_k)| \to 0.
%\end{equation*}
%em outras palavras,
%\begin{equation*}
%|\Phi(g_{j_k} - |f|)| \to 0
%\end{equation*}
%para todo $\Phi \in E^{**}$ e assim
%\begin{equation*}
%g_{j_k} \xrightarrow{w} |f|.
%\end{equation*}
Since, $T$ is $\sigma$-unbounded weak Dunford-Pettis then
\begin{equation*}
g_k(|T(x_{j_k})|\wedge u) \to 0
\end{equation*}
that is
\begin{equation*}
\||T(x_{j_k})|\wedge u\| \to 0.
\end{equation*}
Which does not occur.
\end{proof}

\begin{corollary}
Let $E$ and $F$ be Banach lattices and $T \in \mathcal{L}(E; F)$ with $F$ is a reflexive and $F^*$ has weakly sequentially continuous  lattices operations. If $0 \le S \le T$ and $T \in \mathcal{L}_{\sigma-uwDP}(E; F)$ then $S \in \mathcal{L}_{\sigma-uDP}(E; F)$.
\end{corollary}

%\begin{proof}
%Do Teorema acima $T \in \mathcal{L}_{\sigma-uDP}(E; F)$ e por \cite[Proposition $3.1$]{OCCW19} temos que $S \in \mathcal{L}_{\sigma-uDP}(E; F)$.
%\end{proof}

\begin{definition}
Let $E$ and $F$ be Banach lattices and $T \in \mathcal{L}(E; F)$.
\begin{enumerate}
\item[$(a)$] $T$  is called
unbounded absolute weak continuous operator (uaw-continuous, for short) whenever $T(x_{\alpha}) \xrightarrow{uaw} 0$ in $F$ for any bounded $x_{\alpha}  \xrightarrow{uaw} 0$ in $E$. In this case, we write $T \in \mathcal{L}_{uaw}(E; F)$.
\item[$(b)$] $T$ is called
$\sigma$-unbounded absolute weak continuous operator ($\sigma$-uaw-continuous, for short) whenever $T(x_{j})  \xrightarrow{uaw} 0$ in $F$ for any bounded $x_{j}  \xrightarrow{uaw} 0$ in $E$.  In this case, we write $T \in \mathcal{L}_{\sigma-uaw}(E; F)$.
\end{enumerate}
\end{definition}
%\begin{definition}
%An continuous operator $T$ from a Banach lattice $E$ into a Banach lattice $F$ is called
%unbounded absolute weak continuous operator (uaw-continuous, for short) whenever $T(x_{\alpha}) \xrightarrow{uaw} 0$ in $F$ for any bounded $x_{\alpha}  \xrightarrow{uaw} 0$ in $E$. In this case, we write $T \in \mathcal{L}_{uaw}(E; F)$.
%\end{definition}
%and, in the sequential case
%\begin{definition}
%An continuous operator $T$ from a Banach lattice $E$ into a Banach lattice $F$ is called
%$\sigma$-unbounded absolute weak continuous operator ($\sigma$-uaw-continuous, for short) whenever $T(x_{j})  \xrightarrow{uaw} 0$ in $F$ for any bounded $x_{j}  \xrightarrow{uaw} 0$ in $E$.  In this case, we write $T \in \mathcal{L}_{\sigma-uaw}(E; F)$.
%\end{definition}
For more details about this class, see \cite{Z21}.
\begin{remark}\label{uawinun}\
\begin{enumerate}
\item[$(a)$] It is immediate that for any Banach lattices $E$ and $F$, $T \in \mathcal{L}_{uaw}(E; F)$ if, and only if, $T(x_{\alpha}) \xrightarrow{uaw} T(x)$ in $F$ whenever $x_{\alpha} \xrightarrow{uaw} x$ in $E$.
\item[$(b)$] Let $E$ and $F$ be Banach lattices and $F$ has the Schur property. Then,
\begin{equation*}
\mathcal{L}_{\sigma-uaw}(E; F) \subset \mathcal{L}_{\sigma-un}(E; F).
\end{equation*}
\end{enumerate}

\end{remark}
\begin{definition}
Let $E$ be a Banach lattice. We say that $E$ has the $\sigma$-unbounded Dunford-Pettis property if, for every $(x_j)_{j=1}^{\infty}$ in $E$ and $(f_j)_{j=1}^{\infty}$ in $(E^*)^+$ with $x_j \xrightarrow{uaw} 0$ in $E$ and $f_j \xrightarrow{w} 0$ in $E^*$ we have
\begin{equation*}
|f_j(|x_j|\wedge u)| \to 0
\end{equation*}
for any $u \in E^+$ and $j \to \infty$.
\end{definition}

It is immediate to see that if the Banach lattice $E$ has the Dunford-Pettis property, then it also has the $\sigma$-unbounded Dunford-Pettis property. 

We can also observe that every Banach lattice $E$ which $E^*$ has weakly sequentially  continuous lattice operations has the $\sigma$-unbounded Dunford-Pettis property. Indeed, for any $(x_j)_{j=1}^{\infty}$ in $E$ with $x_j \xrightarrow{uaw} 0$ and $(f_j)_{j=1}^{\infty}$ in $(E^*)^+$ with $f_j \xrightarrow{w} 0$. Note that, for any $u \in E^+$
\begin{equation*}
0 \le |x_j|\wedge u \le u \Rightarrow 0 \le f_j(|x_j|\wedge u) \le f_j(u).
\end{equation*}
Since, $f_j(u) \to 0$, then
\begin{equation*}
f_j(|x_j|\wedge u) \to 0.
\end{equation*}
%In particular, for every $(x_j)_{j=1}^{\infty}$ in $E$ with $x_j \xrightarrow{uaw} 0$ we have
%\begin{equation*}
%f_j(|x_j|\wedge u) \to 0.
%\end{equation*}
Now consider $(x_j)_{j=1}^{\infty}$ in $E$ and and $(f_j)_{j=1}^{\infty}$ in $E^*$ with $x_j \xrightarrow{uaw} 0$  and $f_j \xrightarrow{w} 0$.Since $E^*$ has weakly sequentially  continuous lattice operations then $|f_j| \xrightarrow{w} 0$, then
\begin{equation*}
|f_j(|x_j|\wedge u)| \le |f_j|(|x_j|\wedge u)
\end{equation*}
and hence,
\begin{equation*}
|f_j(|x_j|\wedge u)|  \to 0,
\end{equation*}
as $j \to \infty$. 

It is easy to see that every Banach lattice with the Schur property has weakly sequentially continuous lattice operations. Consequently, every Banach lattice whose dual has the Schur property possesses the $\sigma$-unbounded Dunford-Pettis property. Thus, $c_0$ has the $\sigma$-unbounded Dunford-Pettis property because $\ell_1 = c_0^*$ has the Schur property (see Schur’s Theorem  \cite[Theorem $1.7$]{DJT95}).

%and this implies that $f_j^{+} \xrightarrow{w} 0$ and $f^{-} \xrightarrow{w} 0$.  Which results in the stated sentence.

%\begin{definition}
%Let $E$ be a Banach lattice. 
%\begin{enumerate}
%\item[$(a)$] We say that $E$ has the unbounded Dunford-Pettis property if, for evey $(x_{\alpha})_{\alpha}$ in $E$ and $(f_{\alpha})_{\alpha}$ in $E^*$ with $x_{\alpha} \xrightarrow{uaw} 0$ in $E$ and $f_{\alpha} \xrightarrow{w} 0$ in $E^*$ we have
%\begin{equation*}
%|f_{\alpha}(|x_{\alpha}|\wedge u)| \to 0
%\end{equation*}
%for any $u \in E^+$.
%\item[$(b)$] We say that $E$ has the $\sigma$-unbounded Schur property if, for evey $(x_j)_{j=1}^{\infty}$ in $E$ and $(f_j)_{j=1}^{\infty}$ in $E^*$ with $x_j \xrightarrow{uaw} 0$ in $E$ and $f_j \xrightarrow{w} 0$ in $E^*$ we have
%\begin{equation*}
%|f_j(|x_j|\wedge u)| \to 0
%\end{equation*}
%for any $u \in E^+$ and $j \to \infty$.
%\end{enumerate}
%\end{definition}

Following the same spirit as Proposition \ref{unuDP}, we can obtain the following result.

\begin{proposition}
Let $E$ be a Banach lattice. Then the following statements are equivalent:
\begin{enumerate}
\item[$(a)$] Every uaw-continuous opererator $T \in \mathcal{L}(E; E)$ is a $\sigma$-unbounded weak Dunford-Pettis;
\item[$(b)$] The identity $Id_E \colon E \to E$ is a $\sigma$-unbounded weak Dunford-Pettis operator;
\item[$(c)$] $E$ has the $\sigma$-unbounded Dunford-Pettis property.
\end{enumerate}
\end{proposition}

%\textcolor{red}{An interesting point is that, for any $E$ and $F$ be Dedekind complete Banach lattices and $T\colon E \to F$ be order bounded (see, \cite[Section $1.3$]{MN91}), it always holds that $$|T| \in \mathcal{L}_{uwDP}(E; F) \iff T^+ \text{ and } T^- \in \mathcal{L}_{uwDP}(E; F).$$ Indeed, let $(x_{\alpha})_{\alpha}$ in $E$ and $(f_{\alpha})_{\alpha}$ in $(F^*)^+$ with $x_j \xrightarrow{uaw} x$ in $E$ and $f_{\alpha} \xrightarrow{w} f$ in $F^*$. It is immediate from $|T| = f^+ + f^-$ that, if $T^+$ and $T ^-$ are unbounded Dunford-Pettis then $|T|$ is unbounded Dunford-Pettis.}

\begin{proposition}\label{uwDPisuaw}
Let $E$ and $F$ be Banach lattices. Then,  every unbounded weak Dunford-Pettis $T\colon E \to F$ is unbounded absolutely weak continuous operator.
\end{proposition}

\begin{proof}
Let $T \in \mathcal{L}_{uwDP}(E; F)$ and consider $(x_{\alpha})_{\alpha}$ in $E$ with $x_{\alpha} \xrightarrow{uaw} 0$. For any $f \in (F^*)^+$, consider $f_{\alpha} = f$, for every $\alpha$. Thus, obviously $f_{\alpha} \xrightarrow{w} f$. Since, $T \in \mathcal{L}_{uwDP}(E; F)$ we have
\begin{equation*}
f(|T(x_{\alpha}|\wedge u) \to 0
\end{equation*}
for every $u \in F^+$. Now, for any $g \in F^*$ we can write $g = g^+-g^-$ with $g^+, g^- \in (F^*)^+$, thus
\begin{equation*}
g(|T(x_{\alpha})|\wedge u) = g^+(|T(x_{\alpha})|\wedge u) - g^-(|T(x_{\alpha})|\wedge u) \to 0
\end{equation*}
for every $u \in F^+$. So, $|T(x_{\alpha})|\wedge u \xrightarrow{w} 0$. Hence, $T \in \mathcal{L}_{uaw}(E; F)$.

\end{proof}

\begin{theorem}
Let $E$ and $F$ be Banach lattices. If $E$ has norm order continuous and $F$ has the Schur property then
\begin{equation*}
\mathcal{L}_{\sigma-uwDP}(E; F) = \mathcal{L}_{\sigma-uaw}(E; F) = \mathcal{L}_{\sigma-un}(E; F).
\end{equation*}
\end{theorem}

\begin{proof}
It follows from Remark \eqref{uawinun} and Proposition \eqref{uwDPisuaw}  that
\begin{equation*}
\mathcal{L}_{\sigma-uwDP}(E; F) \subset \mathcal{L}_{\sigma-uaw}(E; F) \subset \mathcal{L}_{\sigma-un}(E; F)
\end{equation*}
whenever $E$ end $F$ are Banach lattices and $F$ has the Schur property.  Let's see then that 
\begin{equation*}
\mathcal{L}_{\sigma-un}(E; F) \subset \mathcal{L}_{\sigma-uwDP}(E; F).
\end{equation*}
Let $T \in \mathcal{L}_{un}(E; F)$ and $(x_j)_{j=1}^{\infty}$ in $E$ with $x_j \xrightarrow{uaw} x$ in $E$. Considere also $(f_j)_{j=1}^{\infty}$ in $(F^*)^+$ with $f_j \xrightarrow{w} f$ in $F^*$. As $E$ has norm order continuous, by \cite[Theorem $4$]{Z18}, $x_j \xrightarrow{un} x$. So, $T(x_j) \xrightarrow{un} T(x)$. Then, for every $u \in F^+$
\begin{align*}
|f_j(|T(x_j)|\wedge u) - f(|T(x)|\wedge u)| &\le |f_j\left( (|T(x_j)| - |T(x)|)\wedge u \right)| + |(f_j - f)(|T(x)|\wedge u)|\\
&\le \|f_j\| \||T(x_j) - T(x)|\wedge u\| + |(f_j - f)(|T(x)|\wedge u)|.
\end{align*}
So, 
\begin{equation*}
f_j(|T(x_j)|\wedge u) \to f(|T(x)|\wedge u)
\end{equation*}
for any $u \in F^+$. Hence, $T \in \mathcal{L}_{\sigma-uwDP}(E; F)$.
\end{proof}

It is important to note that for every Banach lattices $E, F$ and $G$, if $S \in \mathcal{L}_{uaw}(E; F)$ and $T \in \mathcal{L}_{uwDP}(F; G)$, then $T\circ S \in \mathcal{L}_{uwDP}(E; G)$. Indeed, let $(x_{\alpha})_{\alpha}$ in $E$ and $(f_{\alpha})_{\alpha}$ in $(G^*)^+$ with $x_{\alpha} \xrightarrow{uaw} x$ in $E$ and $f_{\alpha} \xrightarrow{w} f$ in $G^*$. Since $x_{\alpha} \xrightarrow{uaw} x$ then $S(x_{\alpha}) \xrightarrow{uaw} S(x)$ and, for any $u \in G^+$
\begin{equation*}
f_{\alpha}(|T\circ S(x_{\alpha})|\wedge u) - f(|T\circ S(x)| \wedge u) = f_{\alpha}(|T(S(x_{\alpha}))|\wedge u) - f(|T(S(x))| \wedge u) \to 0.
\end{equation*}
Hence, $T\circ S \in \mathcal{L}_{uwDP}(E; G)$.
We can also easily see that if $S \in \mathcal{L}_{uaw}(F; G)$ and $T \in \mathcal{L}_{uDP}(E; F)$ then $S \circ T \in \mathcal{L}_{uDP}(E; G)$.

\end{document}